\documentclass[12pt]{article}
\usepackage[T1]{fontenc}
\usepackage[utf8]{inputenc}
\usepackage{geometry}
\usepackage{xcolor}
\usepackage{graphicx}
\usepackage{amsmath,amssymb}
\usepackage{amsmath,mathrsfs}
\usepackage{mathtools}
\usepackage{proof}
\usepackage{amsfonts}
\usepackage{amssymb}
\usepackage{stmaryrd}
\usepackage{pstricks}
\usepackage{setspace}
\usepackage[english]{varioref}
\usepackage{placeins}
\usepackage[english]{babel}
\usepackage{tikz}
\usepackage{makeidx}
\graphicspath{{images/}}
\usepackage[babel=true]{csquotes}
\newcommand{\comment}[1]{}
\newif\ifpdf
\ifpdf \else \fi \textwidth = 6.5 in \textheight = 9 in
\usepackage{bbm}
\usepackage{thmtools}
\usepackage{amsthm}
\usepackage{babel}

\usepackage{mathtools}

\newenvironment{proof}[1][Proof]{\noindent\textbf{#1.} }{$\Box$}

\providecommand{\U}[1]{\protect\rule{.lin}{.lin}}

\newtheorem{theorem}{Theorem}

\newtheorem{corollary}[theorem]{Corollary}

\theoremstyle{definition}

\newtheorem{observation}[theorem]{Observation}

\newtheorem{proposition}[theorem]{Proposition}

\usepackage{caption}
\usepackage{subcaption}

\begin{document}

\title{Sub-quorum colorings of some infinite families of caterpillars}

\author{Rafik Sahbi$^{1}$, Youcef Belkina$^{2}$ and Amar Bennadji$^{2}$\\$^{1}${\small Department of the Preparatory Training}\\{\small National Higher School of Advanced Technologies}\\{\small B.P. 474, Martyrs Square, Algiers 16001, Algeria.}\\{\small E-mail: r.sahbi@g.essa-alger.edu.dz}\\$^{2}${\small Department of Mathematics, University of Blida 1}\\{\small E-mails: youcefbelkina@gmail.com}\\{\small bennadjitalout@gmail.com}}
\date{}
\maketitle

\begin{abstract}
A partition $\pi=\{V_{1},V_{2},...,V_{k}\}$ of the vertex set $V$ of a graph $G$ into $k$ color classes $V_{i},$ with $i\in\{1,...,k\}$ is called a {\it quorum coloring} if for every vertex $v\in V,$ at least half of the vertices in the closed neighborhood $N[v]$ of $v$ have the same color as $v.$ The maximum cardinality of a quorum coloring of $G$ is called the {\it quorum coloring number} of $G$ and is denoted by $\psi_{q}(G).$ A {\it sub-quorum coloring} of $G$ is an onto partial function $f:V\rightarrow\left\{1,2,\ldots,\ell\right\}$ having the property that for every vertex $v\in V,$ if $f(v)$ is defined, then at least half of the vertices in $N[v]$ having an image by $f$, have the same color as $v.$ The {\it sub-quorum coloring number} $\psi_{sq}(G)$ equals the maximum value $\ell$ in a sub-quorum coloring of $G.$ In this paper, we determine the exact value of the sub-quorum coloring number for some infinite families of caterpillars including complete $n$-tuple caterpillars and complete caterpillars with minimum spine-vertex degree three.
\newline\newline\textbf{Keywords:} Quorum colorings, defensive alliances, sub-quorum colorings, $k$-independent sets, complete caterpillars.
\newline\textbf{2000 Mathematical Subject Classification:} 05C15, 05C69.
\end{abstract}

\section{Introduction}

All the graphs of this paper are finite and simple, that is, they are undirected, have a finite number of vertices and not multiple edges nor loops.

A partition $\pi=\{V_{1},V_{2},...,V_{k}\}$ of the vertex set $V$ of a graph $G$ into $k$ color classes $V_{i},$ with $i\in\{1,...,k\}$ is called a \textit{quorum coloring} if for every vertex $v\in V,$ at least half of the vertices in $N_{G}[v]$ have the same color as $v.$ The color classes $V_{i}$ are called {\it quorum classes} and each vertex of $V_{i}$ is said to be a {\it quorum vertex}. The quorum class of a vertex $v$ with respect to $\pi$ is denoted by $\mathcal{C}_{\pi}(v).$ The maximum cardinality of a quorum coloring of $G$ is called the \textit{quorum coloring number} of $G$ and is denoted by $\psi_{q}(G).$ A quorum coloring of cardinality $\psi_{q}(G)$ is called a $\psi_{q}$-{\it coloring} of $G.$

Quorum colorings were introduced in 2013 by Hedetniemi, Hedetniemi, Laskar and Mulder \cite{HedQuor} where the authors pointed out that the concept of quorum class coincides with that of defensive alliance that have been extensively studied over the last two decades (cf. \cite{EroAll,FernSurv,Fricke,HayAll,KrisAll,OuazSurv}) and consequently, that the notions of quorum coloring and partition into defensive alliances mean the same thing. However, from the point of view of quorum colorings, each vertex $v$ of a quorum class $V_{i}$ admits $\frac{1}{2}$ as minimum ratio of $|V_{i}|$ and $|N_{G}[v]|,$ while the point of view of partitions into defensive alliances is that the differential of $|V_{i}|$ and $|N_{G}[v]\setminus V_{i}|$ is positive, that is, there are at least as many vertices in $V_{i}$ as in $N_{G}[v]\setminus V_{i}.$ Nevertheless, we will use in this paper the terminology and the notations of \cite{HedQuor}. Note that the authors \cite{HedQuor} suggested the study of a dozen of open problems concerning quorum colorings of graphs about nine of which received answers in the references \cite{SahNewSharp,SahComp,SahRef,SahSol,SahQuor}.

Among the open problems posed in \cite{HedQuor}, there is the study of the following variant of the concept of quorum coloring and its associated invariant.

A {\it sub-quorum coloring} of a graph $G=(V,E)$ is an onto partial function $f:V\rightarrow\left\{1,2,\ldots,\ell\right\}$ having the property that for every vertex $v\in V,$ if $f(v)$ is defined, then at least half of the vertices in $N_{G}[v]$ having an image by $f$, have the same color as $v.$ The {\it sub-quorum coloring number} $\psi_{sq}(G)$ equals the maximum value $\ell$ in a sub-quorum coloring of $G.$ It can be seen from the previous definition that the set $\left\{f^{-1}(i);~i\in\{1,2,\ldots,\ell\right\}\}$ is a quorum coloring of $G\left[f^{-1}\left(\{1,2,\ldots,\ell\}\right)\right].$ Therefore, a sub-quorum coloring of $G$ is nothing but a quorum coloring of an induced subgraph of $G$ by a vertex subset $S\subseteq V.$ Thus, the set $\left\{f^{-1}(i);~i\in\{1,2,\ldots,\ell\right\}\}$ is called the sub-quorum coloring of $G$ associated with $f$ and is denoted by $sq(f),$ and the graph $G\left[f^{-1}\left(\{1,2,\ldots,\ell\}\right)\right]$ is called the subgraph of $G$ associated with the sub-quorum coloring $f$ and is denoted by $G_{f}.$

Quorum colorings and sub-quorum colorings of graphs can be used as a model for data classification problems. Indeed, a data network can be modeled by a graph $G$ whose the vertices represent the data set to classify so that if two vertices are adjacent, then it means that their corresponding data share a minimum number, fixed in advance, of common characteristics. Therefore, we can obtain a refined classification of our data by finding a $\psi_{q}$-coloring of the graph $G$ (cf. \cite{OuazSurv,SahSol,ShaThes}). However, when the quorum coloring number is small with respect to the graph order, this means that our classification is not refined enough, in particular when $\psi_{q}(G)=1.$ To remedy it, we can decide to not take into account some data that compromise a finer clustering by deleting these latter. This can be realized by adopting the sub-quorum colorings model which allows to increase $\psi_{q}(G)$ by finding an induced subgraph of $G$ by a subset $S\in V(G)$ of its vertices  having $\psi_{q}(G[S])>\psi_{q}(G).$ Thereby, a $\psi_{sq}$-coloring of $G$ corresponds to the best possible classification of our initial data set possibly deprived of some of them.

The study of sub-quorum colorings having been posed as an open problem by Hedetniemi {\it et al.} \cite{HedQuor}, we initiate it in this paper by providing the value of the sub-quorum coloring number of some usual families of caterpillars as follows. In Section 2, we recall some standard definitions and terminology of graph theory, in Section 3 we state some basic properties of both quorum and sub-quorum colorings, then in Section 4 we determine the exact value of the sub-quorum coloring number of some infinite classes of caterpillars such as the paths, the stars, the double stars, complete $n$-tuple caterpillars and complete caterpillars with minimum stem degree $3.$

\section{Definitions and terminology}

In this section, we recall some basic definitions and terminology of graph theory. We also define some families of caterpillars with their notations and the notions of $1$ and $2$-independent set.

\noindent Let $G=(V,E)$ be a graph. The vertex set and the edge set of $G$ are also denoted by $V(G)$ and $E(G),$ respectively. The {\it complement} of $G$ is denoted by $\overline{G}.$ For any integer $r\geq2,$ we denote by $rG$ the graph consisting of $r$ disjoint copies of $G.$ The subgraph of $G$ {\it induced} by a subset $S$ of $V$ is denoted by $G[S].$ For every vertex $v\in V$ and every subset $S\subseteq V,$ the \textit{open neighborhood} of $v$ in $S$ is the set $N_{S}(v)=\{u\in S:uv\in E(G)\}$ and its \textit{closed neighborhood} is $N_{S}[v]=N_{S}(v)\cup S;$ in particular, for $S=V$ the sets $N_{S}(v)$ and $N_{S}[v]$ are denoted $N_{G}(v)$ and $N_{G}[v]$ and called the open and closed neighborhood of $v$ in $G,$ respectively. The \textit{degree} of a vertex $v$ in a subset $S\subseteq V$ is is $d_{S}(v)=|N_{S}(v)|;$ in particular, for $S=V$ the degree $d_{V}(v)$ is denoted $d_{G}(v)$ and called the degree of $v$ in the graph $G.$ The maximum and minimum vertex degrees in $G$ are denoted by $\Delta(G)$ and $\delta(G),$ respectively. A vertex of $G$ with degree one is a {\it leaf} or a {\it pendent} vertex of $G.$ The set of all the leaves of $G$ is denoted by $L(G)$ or $L$ if there is no doubt about $G.$ A {\it tree} is a connected graph having no cycle. A {\it caterpillar} is a tree whose the deletion of the leaves results in a path called the {\it spine} of the caterpillar. A spine-vertex is called a {\it stem} if it is adjacent to a leaf. A leaf adjacent to a stem $v$ is a {\it pendant neighbor} of $v.$ A {\it complete caterpillar} is a caterpillar whose each spine-vertex is a stem. The {\it path} and the {\it complete graph} on $m$ vertices are denoted by $P_{m}$ and $K_{m},$ respectively. The {\it corona} $G\circ\overline{K_{n}}$ of a graph $G$ and $\overline{K_{n}}$ is the graph obtained from $G$ by appending $n$ leaves to each vertex of $G.$ A corona graph $P_{m}\circ\overline{K_{n}}$ is a complete caterpillar called {\it complete $n$-tuple caterpillar}; in particular, the caterpillar is said to be {\it simple} if $n=1$ and {\it double} if $n=2.$ The graph $P_{1}\circ\overline{K_{n}}$ is the {\it star} with $n$ leaves denoted by $S_{n},$ whose the unique vertex of $P_{1}$ is called the {\it center} of the star. A {\it double star} $S_{m,n}$ is the graph obtained from the two stars $S_{m}$ and $S_{n}$ by joining their centers.

For a graph $G,$ a subset of vertices $S\subset V(G)$ is called an {\it independent set} or $1${\it-independent set}, if the vertices of $S$ are pairwise non adjacent, that is, no vertex of $S$ admits a neighbor in $S.$ The {\it independence number} of $G,$ denoted $\beta_{1}(G),$ equals the maximum cardinality of an independent set of $G.$ A subset $S\subseteq V(G)$ is called a {\it $2$-independent set} if every vertex of $S$ has at most one neighbor in $S.$ The maximum cardinality of a $2$-independent set of $G$ is called the {\it $2$-independence number} of $G$ and is denoted by $\beta_{2}(G).$

The concepts of $1$-independent set and $2$-independent set are generalizable to the concept of $k$-independent set that has been widely studied in the discrete mathematical literature and whose \cite{ChellFavSurv} constitutes an exhaustive survey for further reading.

\section{Preliminary results}

First, we point out from the definitions of Section 2 that every independent set is a $2$-independent set, which leads us to the following observation.

\begin{observation}\label{Obs1} For any graph $G,$ \[\beta_{1}(G)\leq\beta_{2}(G).\]\end{observation}

Furthermore, concerning the sub-quorum colorings one can immediately deduce the following three direct observations.

\begin{observation}\label{Obs2} For every graph $G,$ \[\displaystyle\psi_{sq}(G)=\max_{S\subseteq V(G)}\psi_{q}(G[S]).\]
\end{observation}

\begin{observation}\label{Obs3} For every graph $G,$ \[\displaystyle\psi_{sq}(G)=\max_{S\subseteq V(G)}\psi_{sq}(G[S]).\]
\end{observation}

\begin{observation}\label{Obs4}\cite{HedQuor} For every graph $G,$ \[\psi_{sq}(G)\geq\psi_{q}(G).\]\end{observation}

Moreover, since each colored vertex having no colored neighbors with respect to a sub-quorum coloring of a graph is a quorum vertex, then by coloring all the vertices of a maximum independent set $S$ of any graph $G$ with $|S|$ different colors, we obtain a sub-quorum coloring of $G$ of cardinality $\beta_{1}(G),$ which implies the following observation.

\begin{observation}\label{Obs5}\cite{HedQuor} For every graph $G,$ \[\displaystyle\psi_{sq}(G)\geq\beta_{1}(G).\]
\end{observation}

More generally, if a colored vertex $v$ has at most one colored neighbor not having the color of $v,$ then $v$ is a quorum vertex. Thus, by coloring all the vertices of a $2$-independent set $S$ of a graph $G$ with $|S|$ different colors, we obtain a sub-quorum coloring of cardinality $\beta_{2}(G).$ This gives the following observation.

\begin{observation}\label{Obs6} For every graph $G,$ \[\displaystyle\psi_{sq}(G)\geq\beta_{2}(G).\]
\end{observation}

Let us now state some fundamental results on quorum colorings thanks to which we will deduce analogous basic properties for sub-quorum colorings. The first one says that the quorum coloring number is a linear function with respect to the disjoint union of graphs.

\begin{proposition}\label{Pro7}\cite{EroAll} Let $G$ be a disconnected graph whose components are $G_{1},G_{2},\ldots,G_{r}$ ($r\geq1$). Then \[\displaystyle\psi_{q}(G)=\sum_{1\leq i\leq r}\psi_{q}(G_{i}).\]\end{proposition}

The second result establishes that every quorum class of a quorum coloring is connected.

\begin{proposition}\label{Pro8}\cite{HedQuor} Let $G$ be a graph, and let $\pi=\{V_{1},V_{2},\ldots,V_{k}\}$ be any $\psi_{q}$-coloring of $G.$ Then, for every $i,$ $1\leq i\leq k,$ the induced subgraph $G[V_{i}]$ is connected.\end{proposition}

From Propositions \ref{Pro7} and \ref{Pro8} and Observation \ref{Obs2}, we deduce the following properties for the sub-quorum colorings.

\begin{proposition}\label{Pro9} Let $G$ be a disconnected graph whose components are $G_{1},G_{2},\ldots,G_{r}$ ($r\geq1$). Then \[\displaystyle\psi_{sq}(G)=\sum_{1\leq i\leq r}\psi_{sq}(G_{i}).\]\end{proposition}

\begin{proposition}\label{Pro10} Let $G$ be a graph, and let $\pi=\{V_{1},V_{2},\ldots,V_{k}\}$ be any $\psi_{sq}$-coloring of $G.$ Then, for every $i,$ $1\leq i\leq k,$ the induced subgraph $G[V_{i}]$ is connected.\end{proposition}

The next result, due to Eroh and Gera \cite{EroAll} provides a sharp lower and upper bounds for the quorum coloring number of an arbitrary graph of order at least three, that are respectively attained by the complete graphs of odd order and the complements of the complete graphs.

\begin{proposition}\label{Pro11}\label{EroAll} \end{proposition} Let $G$ be a graph of order $n\geq3.$ Then \[1\leq\psi_{q}(G)\leq n.\]

Finally, Observations \ref{Obs1}, \ref{Obs4}, \ref{Obs5} and \ref{Obs6} imply the following proposition.

\begin{proposition}\label{Pro12} Let $G$ be a graph. Then, \[\psi_{sq}(G)\geq\max\left\{\psi_{q}(G),\beta_{2}(G)\right\}.\] \end{proposition}

\section{Sub-quorum colorings of some infinite classes of trees}

In this section, we prove theorems providing the exact value of the sub-quorum coloring number of some infinite families of trees. In fact, we first determine the sub-quorum coloring number of paths, stars and double stars then use these results to deduce that of $n$-tuple complete caterpillars and complete caterpillars with minimum stem degree $3.$ In what follows, we will denote a sub-quorum-coloring of a graph either by a function or by a set, whichever is more convenient.

\subsection{Paths}

The first class for which we determine the sub-quorum coloring number is the simple class of paths whose a $\psi_{sq}$-coloring is illustrated in Figure~\ref{1}. Formally, we have the following proposition.

\begin{proposition}\label{Pro13} For every integer $n\geq1,$
\[\psi_{sq}(P_{n})=\beta_{2}(P_{n})=\left\lceil\dfrac{2n}{3}\right\rceil\]\end{proposition}

\begin{proof} For $n\in \{1,2\},$ it is obvious that $\psi_{sq}(P_{n})=n,$
and so that the result holds since one can easily check in this case that $n=\left\lceil\dfrac{2n}{3}\right\rceil.$ Suppose now that $n\geq3$ and let $\pi$ be a sub-quorum coloring of $P_{n}=v_{1}v_{2}\ldots v_{n}.$ On the one hand, observe that for every integer $i\in\{1,\ldots,n-2\},$ the maximum number of colors that one can use to color three consecutive vertices $v_{i},$ $v_{i+1}$ and $v_{i+2}$ of $P_{n}$ does not exceed $2$ for otherwise, $v_{i+1}$ would not be a quorum vertex. Then, $\psi_{sq}(P_{n})\leq\left\lceil\dfrac{2n}{3}\right\rceil.$ On the other hand, by setting $\pi=\left\{\{v_{i}\}~|~1\leq i\leq n\text{~and~}i\not\equiv0[3]\right\},$ we obtain that $\pi$ is a sub-quorum coloring of $P_{n}$ since it is clearly a $2$-independent set of $P_{n}.$ Consequently, $\psi_{sq}(P_{n})\geq|\pi|=\left\lceil\dfrac{2n}{3}\right\rceil.$ As a result, $\psi_{sq}(P_{n})=\beta_{2}(P_{n})=\left\lceil\dfrac{2n}{3}\right\rceil.$ \end{proof}

\vspace{0.3cm}

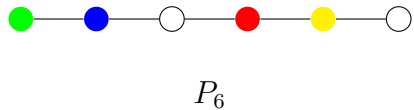
\begin{figure}[htbp]
\begin{center}
\begin{tikzpicture}[inner sep=1.15mm]
\tikzstyle{a}=[circle,fill=green!100]
\tikzstyle{b}=[circle,fill=red!100]
\tikzstyle{d}=[circle,fill=blue!100]
\tikzstyle{p}=[circle,fill=yellow!100]
\tikzstyle{i}=[circle,fill=black!100]
\tikzstyle{w}=[circle,fill=black!20]
\tikzstyle{n}=[rectangle,fill=black!0]
\tikzstyle{j}=[shape=circle,draw]
\tikzstyle{e}=[-]
\node [a](v1)at (0,0){};
\node [d](v2)at (1,0){};
\node [j](v3)at (2,0){};\node [n](v7)at (2.5,-1){$ P_{6} $};
\node [b](v4)at (3,0){};
\node [p](v5)at (4,0){};
\node [j](v6)at (5,0){};
\draw[e](v1)--(v2);\draw[e](v2)--(v3);\draw[e](v3)--(v4);\draw[e](v4)--(v5);
\draw[e](v5)--(v6);
\end{tikzpicture}
\caption{A $\psi_{sq}$-coloring of a path $P_{6}$}
\label{1}
\end{center}
\end{figure}

\subsection{Stars}

A star with at least two pendant vertices has a $\psi_{sq}$-coloring whose each of its classes consists in a single leaf (Figure~\ref{2}), as shown through the following proposition.

\begin{proposition}\label{Pro14} For every integer $n\geq2,$ \[\psi_{sq}(S_{n})=\beta_{2}(S_{n})=n.\]\end{proposition}

\begin{proof} Let $S_{n}$ be a star of center $v$ and let $L$ be the set of its $n$ leaves, with $n\geq2.$ Since $d_{S_{n}}(v)\geq2,$ then we have necessarily $\psi_{sq}(S_{n})\leq n$ for otherwise, $v$ would not be a quorum vertex. Furthermore, $L$ is clearly a maximum $2$-independent set of $S_{n}$ of cardinality $n.$ Therefore, we obtain by Observation \ref{Obs5} that $\psi_{sq}(S_{n})\geq\beta_{2}(S_{n})=n.$ Hence the proposition.
\end{proof}

\vspace{0.3cm}

\begin{figure}[htbp]
\begin{center}
\begin{tikzpicture}[inner sep=1.15mm]
\tikzstyle{a}=[circle,fill=green!100]
\tikzstyle{b}=[circle,fill=red!100]
\tikzstyle{d}=[circle,fill=blue!100]
\tikzstyle{p}=[circle,fill=yellow!100]
\tikzstyle{i}=[circle,fill=black!100]
\tikzstyle{j}=[shape=circle,draw]
\tikzstyle{w}=[circle,fill=brown!100]
\tikzstyle{n}=[rectangle,fill=black!0]
\tikzstyle{e}=[-]
\node [j](v1)at (3,3){};
\node [a](v2)at (3,2){};
\node [w](v3)at (2,2){};
\node [b](v4)at (4,2){};
\node [d](v5)at (1,2){};
\node [p](v6)at (0,2){};
\draw[e](v1)--(v2);\draw[e](v1)--(v3);\draw[e](v1)--(v4);\draw[e](v1)--(v5);\draw[e](v1)--(v6);
\end{tikzpicture}
\caption{A $\psi_{sq}$-coloring of the star $S_{5}$}
\label{2}
\end{center}
\end{figure}
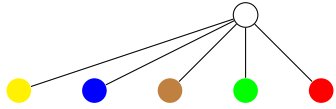

\subsection{Double stars}

The sub-quorum coloring of the double stars is given by the following theorem.

\begin{theorem}\label{The15} For every integers $m,n\geq1,$ \[\psi_{sq}(S_{m,n})=\left\{\begin{array}{cc}
m+n+1, & \text{~if~} \min{\{m,n\}}=1 \text{~or~} \max{\{m,n\}}\leq2, \\
\beta_{2}(S_{m,n})=m+n & \text{~if~} \min{\{m,n\}}\geq2 \text{~and~} \max{\{m,n\}}\geq3.\end{array}\right.\]\end{theorem}

\begin{proof} Set $G=S_{m,n},$ $V(G)=V,$ $L(G)=\{v_{1},v_{2},\ldots,v_{m+n}\}$ and $C=V\setminus L(G)=\{u_{0},v_{0}\}$ with $d_{G}(u_{0})\leq d_{G}(v_{0})=\Delta(G).$ Consider the two following cases.
\newline\newline{\bf Case 1.} $\min{\{m,n\}}=1 \text{~or~} \max{\{m,n\}}\leq2.$ Let $\pi$ be a quorum coloring of $G.$ Since $d_{G}(u_{0})\geq\min\{m,n\}+1\geq 2,$ then we have $|\mathcal{C}_{\pi}(u_{0})|\geq\left\lceil\frac{2+1}{2}\right\rceil=2.$ It follows that \[\psi_{q}(G)\leq1+|V\setminus\mathcal{C}_{\pi}(u_{0})|\leq1+m+n+2-2=m+n+1.~~~~~~~~(1)\] Furthermore, for every non empty subset $S\subseteq V,$ we have by Proposition \ref{Pro11} that \[\psi_{q}(G[V\setminus S])\leq|V\setminus S|\leq m+n+1.~~~~~~~~(2)\] Thus, we deduce from Inequalities $(1)$ and $(2)$ that \[\psi_{sq}(G)\leq m+n+1.\] Now, we are going to prove that $\psi_{sq}(G)\geq m+n+1$ by considering the two following subcases.
\newline{\bf Subcase 1.1.} $m=n=2.$ Let $f$ be the function defined on $V$ by
\[f(v)=\left\{\begin{array}{cc}
i, & \text{~if~} v=v_{i} \text{~for~} i\in\{1,2,\ldots,m+n\},\\
m+n+1 & \text{~if~}v\notin L(G).\end{array}\right.\]
One can see without difficulty that $f$ is a sub-quorum coloring of $G,$ which implies that $\psi_{sq}(G)\geq|sq(f)|=m+n+1.$
\newline{\bf Subcase 1.2.} $m=1$ or $n=1.$ In this case, we have $d_{G}(u_{0})=2.$ Let $g$ be the function defined on $V\setminus\{v_{0}\}$ by \[f(v)=\left\{\begin{array}{cc}
i, & \text{~if~}v=v_{i} \text{~for~} i\in\{1,2,\ldots,m+n\};\\
m+n+1 & \text{~if~}v=u_{0}.\end{array}\right.\]
Therefore, $g$ is a sub-quorum coloring of $G,$ inducing that $\psi_{sq}(G)\geq|sq(g)|=m+n+1.$
\newline\newline{\bf Case 2.} $\min{\{m,n\}}\geq2$ and $\max{\{m,n\}}\geq3.$ Thereby, we get $d_{G}(v_{0})\geq3.$ Moreover, $L$ is a necessarily a maximum $2$-independent set of $G$ and $\beta_{2}(G)=m+n$ for otherwise, if there exists a $2$-independent set $S$ of $G$ of cardinality at least $m+n+1,$ then one of $u_{0}$ or $v_{0}$ would have at least $3$ neighbors in $S,$ a contradiction. At present, let $h$ be the function defined on $L(G)$ by $h(v)=i,$ with $v=v_{i}$ for some $i\in\{1,2,\ldots,m+n\}.$ So, $h$ is a sub-quorum coloring of $G,$ which leads to $\psi_{sq}(G)\geq|sq(h)|=m+n.$ On the other hand, let $\pi_{1}$ be a quorum coloring of $G.$ Then $|\mathcal{C}_{\pi_{1}}(v_{0})|\geq\left\lceil\frac{4+1}{2}\right\rceil=3,$ which implies that $\psi_{q}(G)\leq1+|V\setminus\mathcal{C}_{\pi_{1}}(v_{0})|\leq1+m+n+2-3=m+n.$ Let $S$ be a nonempty subset of $V.$ Let us show that $\psi_{q}(G[V\setminus S])\leq m+n$ with respect to the following two subcases.
\newline{\bf Subcase 2.1.} $|S|=1.$ In this case, $\Delta(G[V\setminus S])\geq2.$ Let $w_{0}$ be a vertex of $V\setminus S$ such that $d_{V\setminus S}(w_{0})=\Delta(G[V\setminus S]).$ Therefore, by arbitrarily choosing a quorum coloring $\pi_{2}$ of $G[V\setminus S],$ we obtain that $|\mathcal{C}_{\pi_{2}}(w_{0})|\geq\left\lceil\frac{2+1}{2}\right\rceil=2,$ then that $\psi_{q}(G[V\setminus S])\leq1+|V\setminus\mathcal{C}_{\pi_{2}}(w_{0})|\leq1+m+n+1-2=m+n.$
\newline{\bf Subcase 2.2.} $|S|\geq2.$ Here, we have $\psi_{q}(G[V\setminus S])\leq|V\setminus S|\leq m+n.$ \\
Thus, for every nonempty subset $S\subseteq V,$ we get $\psi_{sq}(G[V\setminus S])\leq m+n.$

\vspace{0.3cm}\noindent The proof is complete. \end{proof}\\

The $\psi_{sq}$-colorings of the double stars are illustrated in Figure~\ref{3}.

\vspace{0.3cm}

\begin{figure}[htbp]
\begin{center}
\begin{tikzpicture}[inner sep=1.1mm]
\tikzstyle{a}=[circle,fill=black!60]
\tikzstyle{b}=[circle,fill=green!30]
\tikzstyle{c}=[circle,fill=green!90]
\tikzstyle{d}=[circle,fill=red!30]
\tikzstyle{f}=[circle,fill=orange!1000]
\tikzstyle{g}=[circle,fill=brown!100]
\tikzstyle{h}=[circle,fill=blue!100]
\tikzstyle{l}=[circle,fill=yellow!100]
\tikzstyle{m}=[circle,fill=green!100]
\tikzstyle{n}=[circle,fill=black!100]
\tikzstyle{o}=[circle,fill=brown!100]
\tikzstyle{p}=[circle,fill=brown!100]
\tikzstyle{k}=[circle,fill=green!100]
\tikzstyle{w}=[circle,fill=black!20]
\tikzstyle{r}=[circle,fill=pink!100]
\tikzstyle{s}=[circle,fill=pink!30]
\tikzstyle{j}=[shape=circle,draw]
\tikzstyle{e}=[-]
\tikzstyle{c}=[draw,dashed]
\tikzstyle{zz}=[rectangle,fill=black!0]
\node [j](v1)at (0,0){};
\node [j](v2)at (2,0){};
\node [g](v3)at (-1,0){};
\node [h](v4)at (3,0){};\node [zz](v15)at (1,-2){$S_{3,3}$};
\node [l](v5)at (-1,1){};
\node [m](v6)at (-1,-1){};
\node [n](v7)at (3,1){};
\node [f](v8)at (3,-1){};
\node [g](v9)at (6,0){};
\node [g](v10)at (8,0){};
\node [h](v11)at (5,1){};\node [zz](v16)at (7,-2){$S_{2,2}$};
\node [l](v12)at (5,-1){};
\node [m](v13)at (9,1){};
\node [n](v14)at (9,-1){};
\draw[e](v1)--(v2);\draw[e](v1)--(v3);
\draw[e](v2)--(v4);\draw[e](v2)--(v7);
\draw[e](v2)--(v8);\draw[e](v1)--(v5);\draw[e](v1)--(v6);
\draw[e](v9)--(v10);\draw[e](v9)--(v11);\draw[e](v9)--(v12);\draw[e](v10)--(v13);\draw[e](v10)--(v14);
\end{tikzpicture}
\caption{A $\psi_{sq}$-coloring of $S_{2,2}$ and $S_{3,3}$}
\label{3}
\end{center}
\end{figure}
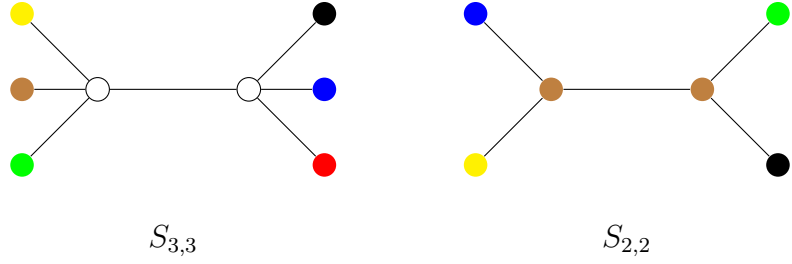

\subsection{Complete caterpillars, complete $n$-tuple caterpillars}

The first result of this subsection concerns the $\psi_{sq}$-colorings of the simple caterpillars (Figure~\ref{4}).

\begin{theorem}\label{The16} For every integer $m\geq1,$ \[\psi_{sq}(P_{m}\circ K_{1})=\beta_{2}(P_{m}\circ K_{1})=m+\left\lceil\dfrac{m}{2}\right\rceil.\]\end{theorem}

\begin{proof} Let $v_{1}v_{2}\ldots v_{m}$ be the spine of $P_{m}\circ K_{1}.$ Set $T_{m}=P_{m}\circ K_{1}$ and $V=V(T_{m}).$ To prove that $\psi_{sq}(T_{m})=m+\left\lceil\dfrac{m}{2}\right\rceil,$ we apply the induction on the integer $m.$ For $m\in\{1,2\},$ we have according to Proposition \ref{Pro13} that $\psi_{sq}(T_{1})=\psi_{sq}(P_{2})=2=1+\left\lceil\dfrac{1}{2}\right\rceil$ and $\psi_{sq}(T_{2})=\psi_{sq}(P_{4})=3=2+\left\lceil\dfrac{2}{2}\right\rceil.$ Suppose the formula true for every integer $l\in\{1,\ldots,m-1\},$ with $m\geq3$ and let $\pi$ be a $\psi_{sq}$-coloring of $T_{m}.$ Set $U=\{v_{m-1},u_{m-1},v_{m},u_{m}\},$ where $u_{j}$ is a pendant neighbor of the stem $v_{j}$ for every $j\in\{m-1,m\}.$ On the one hand, observe that the number of $\pi$'s colors that can be used to color the vertices of $U$ is at most $3$ for otherwise, the stem $v_{m}$ would not be a quorum vertex. Thus, if the stem $v_{m-1}$ is colored $s$ with respect to $\pi,$ then by decoloring $v_{m-1}$ an all the vertices of $T_{m}$ colored $s$ and by assigning three new distinct colors to the vertices $v_{m},$ $u_{m-1}$ et $u_{m},$ we obtain a $\psi_{sq}$-coloring; we can therefore assume that $v_{m-1}$ is not colored with respect to $\pi.$ Since $T_{m}[V\setminus\{v_{m-1}\}]$ is disjoint union of $T_{m}[U\setminus\{v_{m-1}\}]\simeq K_{1}\cup T_{1}$ and $T_{m}[V\setminus U]\simeq T_{m-2},$ then we obtain according to Proposition \ref{Pro9} and the induction hypothesis that \[\psi_{sq}(T_{m})=\psi_{sq}(K_{1}\cup T_{1})+\psi_{sq}(T_{m-2})=m-2+ \left\lceil\dfrac{m-2}{2}\right\rceil+3\]
\hspace{5cm} $=m+1+\left\lceil\dfrac{m}{2}\right\rceil-1=m+\left\lceil\dfrac{m}{2}\right\rceil.$

Lastly, by observing that the set $L\cup\left\{v_{i}~|~1\leq i\leq m\text{ and }i\not\equiv 0[2]\right\}$ is a $2$-independent set of $T_{m}$ of cardinality $m+\left\lceil\dfrac{m}{2}\right\rceil,$ we conclude that $\psi_{sq}(T_{m})=\beta_{2}(T_{m}).$ \end{proof}

\vspace{0.3cm}

\begin{figure}[htbp]
\begin{center}
\begin{tikzpicture}[inner sep=1.1mm]
\tikzstyle{v}=[circle,fill=black!85]
\tikzstyle{a}=[circle,fill=red!100]
\tikzstyle{m}=[circle,fill=yellow!100]
\tikzstyle{s}=[circle,fill=green!100]
\tikzstyle{r}=[circle,fill=blue!100]
\tikzstyle{t}=[circle,fill=pink!100]
\tikzstyle{p}=[circle,fill=orange!100]
\tikzstyle{i}=[circle,fill=brown!100]
\tikzstyle{j}=[shape=circle,draw]
\tikzstyle{e}=[-]
\node [m](v1)at (0,0){};
\node [a](v2)at (0,1){};
\node [j](v4)at (3,0){};
\node [p](v5)at (6,0){};
\node [r](v55)at (6,1){};
\node [s](v44)at (3,1){};
 \draw[e](v1)--(v2);\draw[e](v1)--(v4);
\draw[e](v55)--(v5);\draw[e](v44)--(v4);\draw[e](v4)--(v5);
\end{tikzpicture}
\caption{A $\psi_{sq}-$coloring of the simple caterpillar $P_{3}\circ  K_{1}$}
\label{4}
\end{center}
\end{figure}
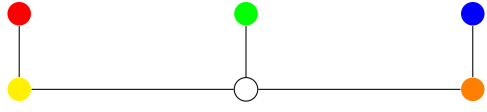

The next result provides the value of the sub-quorum coloring number of the double caterpillars (Figure~\ref{5}).

\begin{theorem}\label{The17} For every integer $m\geq1,$ \[\psi_{sq}(P_{m}\circ\overline{K_{2}})=2m+\left \lfloor\dfrac{1}{2}\left\lceil\dfrac{2m}{3}\right\rceil\right\rfloor.\]\end{theorem}

\begin{proof} Let $v_{1}v_{2}\ldots v_{m}$ be the spine of $P_{m}\circ\overline{K_{2}}.$ Set $T_{m}=P_{m}\circ\overline{K_{2}},$ $V=V(T_{m})$ and $L=V\setminus\{v_{1},v_{2},\ldots,v_{m}\}.$ Let us use the induction on $m.$ For $m=0,$ we set $\psi_{sq}(T_{0})=0=0+\left \lfloor\dfrac{1}{2}\left\lceil\dfrac{0}{3}\right\rceil\right\rfloor.$ For $m=1,$ we have according to Proposition \ref{Pro14} that $\psi_{sq}(T_{1})=\psi_{sq}(S_{2})=2=2+\left\lfloor\dfrac{1}{2}\left\lceil\dfrac{2}{3}\right\rceil\right\rfloor.$ For $m=2,$ by Theorem \ref{The15} we have $\psi_{sq}(T_{2})=\psi_{sq}(S_{2,2})=2+2+1=5=4+\left\lfloor\dfrac{1}{2}\left\lceil\dfrac{4}{3}\right\rceil\right\rfloor.$ Suppose the formula true for every integer $l\in\{0,1,\ldots,m-1\},$ with $m\geq3$ and let $\pi$ be a $\psi_{sq}$-coloring of $T_{m}.$ Set $U=\{v_{m-2},v_{m-1},v_{m}\}\cup\left(N_{T_{m}}\left(\{v_{m-2},v_{m-1},v_{m}\}\right)\cap L\right).$ We will show that the maximum number of colors $c(U)$ of $\pi$ that we can use to color the vertices of $U$ is $7.$ First, remark that $c(U)<9$ for otherwise, no of the stems $v_{m-2},$ $v_{m-1}$ and $v_{m}$ would be a quorum vertex. Assume that $c(U)=8$ and that all the vertices of $U$ are colored. In this case, observe that if two adjacent stems $v_{i}$ and $v_{j}$ of $\{v_{m-2},v_{m-1},v_{m}\}$ are colored differently, then at least one of them would not be a quorum vertex. Consequently, the stems $v_{m-2},$ $v_{m-1}$ and $v_{m}$ have all the same color with respect to $\pi$ and we deduce that $c(U)\leq1+|U\setminus\{v_{m-2},v_{m-1},v_{m}\}|=7,$ a contradiction. Now, suppose that at least one stem $v_{k}$ is not colored for some $k\in\{m-2,m-1,m\}.$ If the vertices of $\{v_{m-2},v_{m-1},v_{m}\}\setminus\{v_{k}\}$ use two distinct colors of $\pi$, then one can see without difficulty that at least one vertex of $\{v_{m-2},v_{m-1},v_{m}\}\setminus\{v_{k}\}$ is not a quorum vertex, a contradiction. We conclude that we have necessarily $c(U)\leq7.$ Moreover, if $v_{m-2}$ is colored with a color $s$ with respect to $\pi,$ then we can decolor it together with all the vertices of $T_{m}$ colored $s,$ assign a new color to the stems $v_{m-1}$ and $v_{m}$ and assign six other distinct new colors to the leaves of $U$ so that we obtain a new $\psi_{sq}$-coloring of $T_{m}$ in which $v_{m-2}$ is not colored; so, we can assume that $v_{m-2}$ is not colored with respect to $\pi.$ Now, since $T_{m}[V\setminus\{v_{m-2}\}]$ is the disjoint union of $T_{m}[V\setminus U]\simeq T_{m-3}$ and $T_{m}[U\setminus\{v_{m-2}\}]\simeq\overline{K_{2}}\cup T_{2},$ then we obtain by Proposition \ref{Pro9} and the induction hypothesis that \[\psi_{sq}(T_{m})= \psi_{sq}(T_{m-3})+\psi_{sq}(\overline{K_{2}}\cup T_{2})=2(m-3)+\left\lfloor\dfrac{1}{2}\left\lceil\dfrac{2(m-3)}{3}\right\rceil
\right\rfloor+2+4+\left\lfloor\dfrac{1}{2}\left\lceil\dfrac{4}{3}\right\rceil
\right\rfloor\]
\[=2m+1+\left\lfloor\dfrac{1}{2}\left(\left\lceil\dfrac{2m}{3}\right\rceil-2\right)\right\rfloor
=2m+1+\left\lfloor\dfrac{1}{2}\left\lceil\dfrac{2m}{3}\right\rceil-1\right\rfloor\]
\hspace{3cm} $=2m+1+\left\lfloor\dfrac{1}{2}\left\lceil\dfrac{2m}{3}\right\rceil\right\rfloor-1
=2m+\left\lfloor\dfrac{1}{2}\left\lceil\dfrac{2m}{3}\right\rceil\right\rfloor.$ \end{proof}

\vspace{0.3cm}

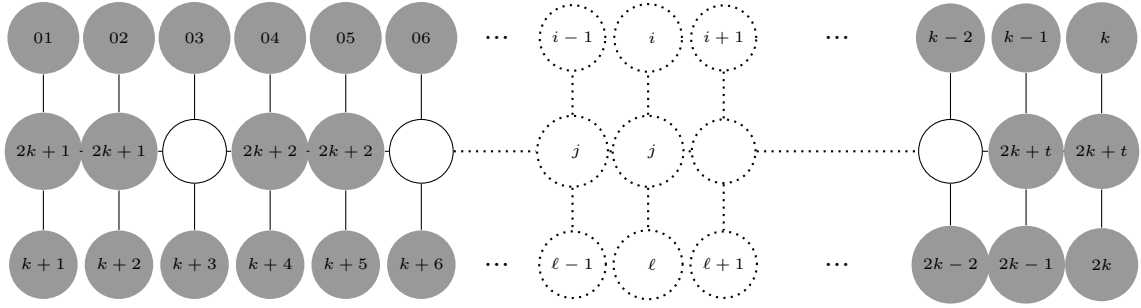
\begin{figure}[htbp]
\begin{center}
\begin{tikzpicture}
\tikzstyle{a}=[circle,fill=black!40,inner xsep=1mm]
\tikzstyle{j}=[shape=circle,draw,inner xsep=1mm]
\tikzstyle{k}=[shape=circle,draw,dotted,thick,inner xsep=1mm]
\tikzstyle{e}=[-]
\tikzstyle{c}=[draw,dotted,thick]
\tikzstyle{d}=[rectangle,fill=black!0]
\node [a](v1)at (0,0){\tiny{$2k+1$}};
\node [a](v2)at (1,0){\tiny{$2k+1$}};
\node [j](v3)at (2,0){\hspace*{0.6cm}};\node [a](v31)at (3,0){\tiny{$2k+2$}};\node [a](v32)at (4,0){\tiny{$2k+2$}};
\node [j](v4)at (12,0){\hspace*{0.6cm}};\node [a](v103)at (3,-1.5){\tiny{$k+4$}};
\node [a](v5)at (13,0){\tiny{$2k+t$}};\node [a](v102)at (3,1.5){\tiny{$\hspace*{0.2cm}04\hspace*{0.2cm}$}};
\node [a](v6)at (14,0){\tiny{$2k+t$}};\node [a](v105)at (4,-1.5){\tiny{$k+5$}};
\node [a](v7)at (0,1.5){\tiny{$\hspace*{0.2cm}01\hspace*{0.2cm}$}};\node [a](v104)at (4,1.5){\tiny{$\hspace*{0.2cm}05\hspace*{0.2cm}$}};
\node [a](v8)at (0,-1.5){\tiny{$k+1$}};\node [j](v33)at (5,0){\hspace*{0.6cm}};
                              \node [a](v37)at (5,-1.5){\tiny{$k+6$}};
\node [a](v9)at (1,1.5){\tiny{$\hspace*{0.2cm}02\hspace*{0.2cm}$}}; \node [a](v344)at (5,1.5){\tiny{$\hspace*{0.2cm}06\hspace*{0.2cm}$}};
\node [a](v10)at (1,-1.5){\tiny{$k+2$}};
\node [a](v11)at (2,1.5){\tiny{$\hspace*{0.2cm}03\hspace*{0.2cm}$}};
\node [a](v12)at (2,-1.5){\tiny{$k+3$}};
\node [a](v13)at (12,1.5){\tiny{$k-2$}};
\node [a](v14)at (12,-1.5){\tiny{$2k-2$}};
\node [a](v15)at (13,1.5){\tiny{$k-1$}};
\node [a](v16)at (13,-1.5){\tiny{$2k-1$}};
\node [a](v17)at (14,1.5){\tiny{$\hspace*{0.3cm}k\hspace*{0.2cm}$}};
\node [a](v18)at (14,-1.5){\tiny{$\hspace*{0.2cm}2k\hspace*{0.2cm}$}};\draw[e](v32)--(v33);
\draw[e](v3)--(v31);\draw[e](v31)--(v32);\draw[e](v31)--(v103);
\draw[e](v31)--(v102);\draw[e](v105)--(v32);\draw[e](v104)--(v32);
\draw[e](v33)--(v37);\draw[e](v33)--(v344);
\draw[e](v1)--(v7);\draw[e](v1)--(v8);
\draw[e](v2)--(v9);\draw[e](v2)--(v10);
\draw[e](v3)--(v11);\draw[e](v3)--(v12);\draw[e](v13)--(v4);\draw[e](v14)--(v4);\draw[e](v5)--(v16);\draw[e](v5)--(v15);\draw[e](v17)--(v6);\draw[e](v6)--(v18);
\draw[e](v1)--(v2);\draw[e](v2)--(v3);
\draw[e](v4)--(v5);\draw[e](v5)--(v6);
\node [k](l1)at (7,1.5){\tiny{$i-1 $}};
\node [k](l2)at (8,1.5){\tiny{$\hspace*{0.3cm}i\hspace*{0.2cm}$}};
\node [k](l3)at (7,-1.5){\tiny{$\ell-1$}};
\node [k](l4)at (8,-1.5){\tiny{$\hspace*{0.3cm}\ell\hspace*{0.2cm}$}};
\node [k](l6)at (9,1.5){\tiny{$i+1$}};
\node [k](l5)at (9,-1.5){\tiny{$\ell+1$}};
\node [k](l7)at (7,0){\tiny{$\hspace*{0.3cm}j\hspace*{0.2cm}$}};
\node [k](l8)at (8,0){\tiny{$\hspace*{0.3cm}j\hspace*{0.2cm}$}};
\node [k](l9)at (9,0){\tiny{$\hspace*{0.6cm}$}};
\draw[c](l7)--(l1);\draw[c](l7)--(l3);\draw[c](l8)--(l2);\draw[c](l8)--(l4);\draw[c](l9)--(l5);\draw[c](l9)--(l6);\draw[c](l9)--(l8);\draw[c](l8)--(l7);\draw[c](v33)--(l7);\draw[c](l9)--(v4);
\node [d](p1)at (6,1.5){...};\node [d](p2)at (6,-1.5){...};\node [d](p3)at (10.5,1.5){...};
\node [d](p4)at (10.5,-1.5){...};
\end{tikzpicture}
\caption{A $\psi_{sq}$-coloring of the double caterpillars}
\label{5}
\end{center}
\end{figure}

The next theorem shows that for the complete caterpillars whose each stem has at least three pendant neighbors, the sub-quorum coloring number corresponds to the number of leaves of such a caterpillar.

\begin{theorem}\label{The18} Let $T_{m}$ be a complete caterpillar with $m$ stems whose each one has at least three pendant neighbors. Then, $$\psi_{sq}(T_{m})=\beta_{2}(T_{m})=|L(T_{m})|.$$
\end{theorem}

\begin{proof} Let $v_{1}\ldots v_{m}$ be the spine of $T_{m}.$ We proceed again by induction. For $m\in\{1,2\},$ one can easily check the validity of the result by applying Proposition \ref{Pro14} and Theorem \ref{The15}. Suppose now that the formula is valid for every integer $l\in\{1,\ldots,m-1\},$ with $m\geq3.$ Set $V=V(T_{m}),$ $L_{m}=N_{T_{m}}(v_{m})\cap L(T_{m})$ and let $\pi$ be a $\psi_{sq}$-coloring of $T_{m}.$ Since $|L_{m}|\geq3,$ then we deduce that the maximum number of $\pi$'s colors used to color the vertices of $\{v_{m}\}\cup L_{m}$ does not exceed $|L_{m}|$ for otherwise, $v_{m}$ would not be a quorum vertex. In addition, if $v_{m}$ is colored with the color $s$ with respect to $\pi,$ then by decoloring it together with all the vertices having the color $s$ and by assigning $|L_{m}|$ new distinct colors to the leaves of $L_{m},$ then we obtain a $\psi_{sq}$-coloring of $T_{m}$ for which $v_{m}$ is not colored. So, we can assume that $v_{m}$ is not colored with respect to $\pi.$ Since $T_{m}[V\setminus\{v_{m}\}]\simeq\overline{K_{n}}$ is the disjoint union of $T_{m}[V\setminus\left(\{v_{m}\}\cup L_{m}\right)]\simeq T_{m-1}$ and $T_{m}[L_{m}]\simeq\overline{K_{|L_{m}|}},$ we obtain that \\

\hspace{4cm} $\displaystyle\psi_{sq}(T_{m})=\psi_{sq}(T_{m-1})+|L_{m}|=\sum_{i=1}^{m}|L_{i}|=|L(T_{m})|.$

Finally, since $L(T_{m})$ is a $2$-independent set of $T_{m},$ we conclude that $\psi_{sq}(T_{m})=\beta_{2}(T_{m}).$ \end{proof}

The following corollary gives the sub-quorum coloring number of the complete $n$-tuple caterpillars with $n\geq3$ as an immediate consequence of Theorem \ref{The18}.

\begin{corollary}\label{Cor19}
\label{Cor19} For every integers $m\geq1$ and $n\geq3,$ \[\psi_{sq}(P_{m}\circ\overline{K_{n}})=\beta_{2}(P_{m}\circ\overline{K_{n}})=mn.\]
\end{corollary}

\section{Conclusion}

We have determined the sub-quorum coloring number of paths, stars and double stars which allowed us to obtain the value of the same parameter for the complete $n$-tuple caterpillars and the complete caterpillars of minimum stem degree $3,$ using a simple induction that relies on the sub-quorum-coloring number of double stars (Theorem \ref{The15}) as induction hypothesis. The interesting fact that emerges from our investigation comes from the inequality chain $\psi_{sq}(G)\geq\beta_{2}(G)\geq\beta_{1}(G)$ (Observations \ref{Obs1} and \ref{Obs2}) which could make it possible to obtain these last two independence parameters through the study of $\psi_{sq}$ as in Theorems \ref{The15},\ref{The16}, \ref{The18} and Corollary \ref{Cor19}. In particular, the characterization of caterpillars attaining the lower bounds $\psi_{q}$ or $\beta_{2}$ are suggested by our investigation. On the other hand, the determination of the sub-quorum coloring number for arbitrary caterpillars is a natural open problem followed from our work. Also, since the $n$-tuple complete caterpillars are corona graphs $P_{m}\circ\overline{K_{n}},$ then the study of the sub-quorum coloring number of the corona graphs $G\circ\overline{K_{n}}$ could be a first step towards a more general study of sub-quorum colorings.

\end{document}